\documentclass[letterpaper, 10 pt, journal, twoside]{IEEEtran}
\IEEEoverridecommandlockouts
\usepackage{amsthm}
\usepackage{amsfonts,amsmath,amssymb}
\usepackage{bm}
\usepackage{graphicx}
\usepackage{color}
\usepackage{balance}
\usepackage{setspace}
\usepackage{balance}
\usepackage{comment}
\usepackage{cases}
\graphicspath{{./fig_eps/}}
\usepackage{here}
\usepackage{array}
\usepackage{float}
\usepackage[linesnumbered, ruled]{algorithm2e}
\usepackage{cite}
\usepackage{subfigure}
\newtheorem{mydef}{Definition}
\newtheorem{mylem}{Lemma}
\newtheorem{mythm}{Theorem}
\newtheorem{mypro}{Problem}

\newtheorem{myas}{Assumption}
\newtheorem{myrem}{Remark}

\newcommand{\rfig}[1]{Fig.\,\ref{#1}} 
 
\newcommand{\req}[1]{\eqref{#1}} 
 
\newcommand{\rtab}[1]{Table\,\ref{#1}}
\newcommand{\rlem}[1]{Lemma\,\ref{#1}}

\newcommand{\rsec}[1]{Section\,\ref{#1}}
\newcommand{\rpro}[1]{Problem\,\ref{#1}}

\newcommand{\rdef}[1]{Definition\,\ref{#1}}
\newcommand{\ras}[1]{Assumption\,\ref{#1}}
\newcommand{\rthm}[1]{Theorem\,\ref{#1}}
\newcommand{\ralg}[1]{Algorithm\,\ref{#1}}

\newcommand{\qedwhite}{\hfill \ensuremath{\Box}}

\begin{document}
\title{{A symbolic approach to the self-triggered design \\ for networked control systems}}

\author{Kazumune~Hashimoto, Adnane Saoud, Masako Kishida, Toshimitsu Ushio and  Dimos~V.~Dimarogonas
\thanks{Kazumune Hashimoto and Toshimitsu Ushio are with the Graduate School of Engineering Science, Osaka University, Osaka, Japan (e-mail: kazumune@kth.se, ushio@sys.es.osaka-u.ac.jp)}
\thanks{Adnane Saoud is with Laboratoire des Signaux et Syst$\grave{e}$mes and with Laboratoire Sp${\rm \acute{e}}$cification V${\rm \acute{e}}$rification, CNRS, Paris-Saclay, France (e-mail: adnane.saoud@l2s.centralesupelec.fr)}
\thanks{Masako Kishida is with National Institute of Informatics (NII), Tokyo, Japan (e-mail: kishida@nii.ac.jp)}
\thanks{Dimos V. Dimarogonas is with the School of Electrical Engineering, KTH Royal Institute of Technology, Sweden (e-mail: dimos@kth.se) }
\thanks{The authors are supported by ERATO HASUO Metamathematics for Systems Design Project (No. JPMJER1603), JST, Labex DigiCosme (project ANR-11-LABEX-0045-DIGICOSME) operated by ANR, the Swedish Research Council (VR), the Swedish Foundation for StrategicResearch (SSF), the H2020 ERC Starting Grant BUCOPHSYS and the Knut och Alice Wallenberg Foundation (KAW).}
}


\maketitle

\thispagestyle{empty}

\begin{abstract}
In this paper, we investigate novel self-triggered controllers for nonlinear control systems with reachability and safety specifications. To synthesize the self-triggered controller, we leverage the notion of symbolic models, or abstractions, which represent abstracted expressions of control systems. The symbolic models will be constructed through the concepts of approximate alternating simulation relations, based on which, and by employing a reachability game, the self-triggered controller is synthesized. We illustrate the effectiveness of the proposed approach through numerical simulations. 
\end{abstract}
\section{Introduction}
\IEEEPARstart{E}{vent} and self-triggered control have been prevalent in recent years as the useful control strategies to reduce communication resources for networked control systems (NCSs)\cite{heemels2012a}. 
The key idea of these approaches is that, network transmissions from sensors to the remote controller are given based on some criteria, such as stability or some control performances. Introducing the event and self-triggered control has been proven effective, since it leads to the potential energy-savings of battery powered devices by mitigating the communication load for NCSs. 

So far, a wide variety of event and self-triggered controllers has been provided from theory to practical implementations, see, e.g., \cite{eventsurvey,event_survey} for survey papers. 
In this paper, we are particularly interested in designing a self-triggered strategy under \textit{reachability} and \textit{safety} specifications. 
In other words, our goal is to design a self-triggered controller, such that the state trajectory enters a target set in finite time (reachability), while at the same time remaining inside a safety set for all times (safety). 
 {To the best of our knowledge, event and self-triggered strategies that can accommodate reachability and safety specifications have been provided only in a few works, see, e.g., 
\cite{evreachable1,evreachable2,hashimoto2018b,hashimoto2017d,evreachable3,hashimoto2019b}.  
Event/self-triggered controllers based on reachability analysis or controlled invariant sets have been proposed in \cite{evreachable1,evreachable2,hashimoto2017d,evreachable3,hashimoto2018b}.  
However, a fundamental assumption required in these previous approaches is that the safety set is \textit{convex}; in some practical applications, such as robot motion planning, safety sets are typically \textit{non-convex} due to the presence of obstacles. Hence, the above previous approaches may be limited for certain practical applications. 
A self-triggered algorithm that deals with non-convex safety sets has been proposed in \cite{hashimoto2019b}. 
In this previous work, a sufficient condition to generate a feasible communication scheduling has been derived, based on the assumption that there exists an \textit{$\delta$-ISS Lyapunov function} for the control system. However, assuming the existence of such a Lyapunov function limits the class of control systems, since it must ensure contractive behaviors between any pair of state trajectories. Therefore, designing an event/self-triggered controller that can accommodate {non-convex} safety constraints and that does not require any stability assumptions is still a challenging problem, which is our main objective and is tackled in this paper.}

In this paper, we investigate a new self-triggered controller that takes different approaches from the previous works in the literature. The main contribution is to employ the notion of \textit{symbolic models} (see, e.g., \cite{tabuadabook2009}). Roughly speaking, the symbolic model represents an abstracted expression of the control system, where each state of the symbolic model corresponds to an aggregate of states of the control system. 
 {The utilization of symbolic models is motivated by the fact that the self-triggered controller for reachability and safety specifications can be synthesized by employing algorithmic techniques from supervisory control, such as reachability/safety games, which, in particular, allow to deal with {non-convex} safety sets. More importantly, the only assumption required for the controller synthesis is Lipschitz continuity, and it does not require any stability assumption such as the one considered in \cite{hashimoto2019b}. 
Hence, the proposed approach is 
advantageous over the afore-mentioned previous works, in the sense that it deals with the non-convexity of safety sets and can be applied to a wide class of (Lipschitz) control systems.}

Our approach is also related to several abstraction schemes, see, e.g., \cite{adnane2018c,arman2017a,girard2009a,adnane2018b,zamani2012a,borri18a,mazo2011,zamani2018a}; in particular, it may be closely related to \cite{adnane2018c,arman2017a}, in which some methods of constructing symbolic models with event-triggered strategies have been provided. Note that our approach differs from these previous results in the following sense. In the previous results, for example in \cite{adnane2018c}, the authors provided a way to construct symbolic models with a \textit{given} event-triggered strategy. On the other hand, our approach aims at \textit{synthesizing} a self-triggered controller through the construction of symbolic models, such that the reachability and safety specifications are fulfilled. 
In addition to \cite{adnane2018c,arman2017a}, several approaches to obtain symbolic models for NCSs have been also proposed, see e.g., \cite{borri18a,zamani2018a}; however, none of these works provided a way to synthesize event/self-triggered controllers, which will be the main objective considered in this paper. 

\smallskip
\noindent 
\textit{Notation.} 
Let $\mathbb{N}$, $\mathbb{N}_{\geq 0}$, $\mathbb{N}_{>0}$, $\mathbb{N}_{a:b}$ be the set of integers, non-negative integers, positive integers, and the set of integers in the interval $[a, b]$, respectively. 
Let $\mathbb{R}$, $\mathbb{R}_{\geq 0}$, $\mathbb{R}_{>0}$ be the set of reals, non-negative reals and positive reals, respectively. We denote by $\| \cdot \|$ the Euclidean norm. 
For given $x \in \mathbb{R}^n$ and $\varepsilon \in \mathbb{R}_{\geq 0}$, let ${\cal B}_{\varepsilon} (x) \subset \mathbb{R}^n$ be the ball set given by ${\cal B}_{\varepsilon} (x) = \{ x' \in\mathbb{R}^n\ |\ \| x' - x \| \leq \varepsilon \}$. For given ${X} \subseteq \mathbb{R}^n$ and $\eta >0$, denote by $[{X}]_\eta \subset \mathbb{R}^n$ the lattice in $X$ with the quantization parameter $\eta$, i.e.,
$[{X}]_\eta =  \{ x\in{X} \ |\ x_i = \frac{2 \eta }{\sqrt{n}} 
a_i,\ a_i\in \mathbb{N},\ i = 1, 2, \ldots, n \}$,   
where $x_i \in \mathbb{R}$ is the $i$-th element of $x$. Given $X\subset \mathbb{R}^n$, let 
$\mathsf{Int}_\varepsilon (X) = \left \{ x \in X\ |\ {\cal B}_{\varepsilon} (x) \subseteq X \right \}$, i.e., $\mathsf{Int}_\varepsilon (X)$ is the set of all states in $X$, such that these are $\varepsilon$-away from the boundary of $X$. Given $X \subseteq \mathbb{R}^n$, denote by $X^*$ the set of all finite sequences of elements in $X$.
Given $x \in \mathbb{R}^n$, $X\subseteq \mathbb{R}^n$, denote by $\mathsf{Nearest}_X (x)$ the closest points in $X$ to $x$, i.e., $\mathsf{Nearest}_X (x) = {\arg\min}_{x' \in X} \| x - x' \|$. Given a set $X$, denote by $2^X$ the power set of $X$ that represents the collection of all subsets of $X$.
\section{Problem formulation}
\subsection{System description}
Consider a networked control system shown in \rfig{NCS}, where the plant and the controller are connected over a communication network. 
We assume that the dynamics of the plant is given by the following nonlinear discrete-time systems: 
\begin{equation}\label{dynamics2}
x_{k+1} = f (x_k, u_k), \ \ x_0 \in X_0,\ u_k \in U
\end{equation}
for all $k\in\mathbb{N}_{\geq 0}$, where $x_k \in \mathbb{R}^{n_x}$ is the state, $u_k\in \mathbb{R}^{n_u}$ is the control input, $X_0 \subset \mathbb{R}^{n_x}$ is the set of initial states, $U \subset \mathbb{R}^{n_u}$ is the set of control inputs, and $f: \mathbb{R}^{n_x} \times \mathbb{R}^{n_u} \rightarrow \mathbb{R}^{n_x}$ is the function that represents the underlying model of the plant. Throughout the paper, we assume that $X_0$ and $U$ are both compact sets. Moreover, we assume that the function $f$ satisfies the following Lipschitz continuity: 
\begin{myas}\label{lipschitz}
The function $f : \mathbb{R}^{n_x} \times \mathbb{R}^{n_u} \rightarrow \mathbb{R}^{n_x}$ is Lipschitz continuous in $x\in\mathbb{R}^{n_x}$, i.e., there exists $L_x \in \mathbb{R}_{\geq 0}$, such that $\| f(x_1, u) - f(x_2, u) \| \leq L_x \| x_1 - x_2\|$ for all $x_1, x_2 \in \mathbb{R}^{n_x}, u\in U$. \qedwhite 
\end{myas}

We say that the sequence $x_0, x_1, x_2, \ldots \in \mathbb{R}^{n_x}$ is a \textit{trajectory} of the system \req{dynamics2}, if $x_0 \in X_0$ and there exist $u_0, u_1, u_2, \ldots \in U$ such that $x_{k+1} = f(x_k, u_k)$, $\forall k\in\mathbb{N}_{\geq 0}$. For simplicity of presentation, we denote by $\phi (x, u, m) \in \mathbb{R}^{n_x}$ the state that is reached from $x \in \mathbb{R}^{n_x}$ with $u \in U$ applied \textit{constantly} for $m$ time steps, i.e., $x^+ = \phi (x, u, m)$ iff there exist $x_{0}, \ldots, x_m \in \mathbb{R}^{n_x}$, such that $x_0 = x$, $x_{k+1} = f(x_k, u)$, $\forall k\in\mathbb{N}_{0:m-1}$, and $x_m = x^+$. The following result is an immediate consequence from \ras{lipschitz}, which will be utilized later in this paper:
\begin{mylem}\label{liplem}
For every $x_{1}, x_{2} \in \mathbb{R}^{n_x}$, $u\in U$, and $m\in\mathbb{N}_{>0}$, $\|\phi (x_1, u, m) -\phi (x_2, u, m) \| \leq L^m _x \| x_1 - x_2 \|$. 
\qedwhite 
\end{mylem}

\subsection{Self-triggered strategy}\label{self_triggered_sec}
Let us now provide the overview of the control strategies. First, let $k_\ell$, $\ell \in \mathbb{N}_{\geq 0}$ with $k_0 = 0$, $k_{\ell+1} > k_\ell$, $\forall \ell \in \mathbb{N}_{\geq 0}$ be the communication time steps when the information is exchanged between the plant and the controller. 
In this paper, we employ a \textit{self-triggered strategy} \cite{heemels2012a}, which means that the controller is defined as a mapping from the state to the corresponding pairs of the control input and the inter-communication time step: 
\begin{align}\label{controller}
C : \mathbb{R}^{n_x} \rightarrow 2^{U \times \mathbb{N}_{> 0}}. 
\end{align}
That is, for each $k_\ell$, $\ell\in\mathbb{N}_{\geq 0}$ the plant transmits the current state information $x_{k_\ell}$ to the controller, and the controller determines both the control input and the inter-communication time step as $\{u_{k_\ell}, m_\ell \} \in C (x_{k_\ell})$. Then, the controller transmits $\{ u_{k_\ell}, m_\ell\}$ to the plant, and the plant applies $u_{k_\ell}$ until $k_{\ell+1} = k_\ell + m_\ell$, i.e., $u_k = u_{k_\ell},\ \forall k \in \mathbb{N}_{k_\ell: k_{\ell+1} -1}$. Then, the next communication is given at $k_{\ell+1}$ and the same procedure as above is iterated. 
 {Given $C$, we say that the sequence $x_0, x_1, x_2, \ldots \in \mathbb{R}^{n_x}$ is a \textit{controlled} trajectory of the system \req{dynamics2} induced by $C$}, if $x_0 \in X_0$, $x_{k+1} = f(x_k, u_{k_\ell})$, $\forall k \in \mathbb{N}_{k_\ell: k_{\ell+1}-1}$, $\forall \ell \in \mathbb{N}_{\geq 0}$, where $k_0 = 0$, $k_{\ell+1} = k_\ell + m_\ell$, $\forall \ell \in \mathbb{N}_{\geq 0}$ and $\{ u_{k_\ell}, m_\ell \} \in C (x_{k_\ell})$, $\forall \ell \in \mathbb{N}_{\geq 0}$. 

\subsection{Problem statement}
Let $X_S \subset \mathbb{R}^{n_x}$ with $X_0 \subseteq X_S$ be the \textit{safety set}, in which the trajectory must remain for all times. In addition, let $X_F \subset \mathbb{R}^{n_x}$ with $X_F \subseteq X_S$ be the \textit{target set}, which the trajectory aims to reach in finite time. 
We assume that $X_S, X_F$ are both compact and can be non-convex sets. 
Specifically, we illustrate the reachability and safety specifications by introducing the notion of \textit{validity} of the controller $C$, which is defined below: 

\begin{mydef}\label{validity}
Given $(X_S, X_F)$, we say that the controller $C$ is \textit{valid} for the system \req{dynamics2} with the specification $(X_S, X_F)$, if for every $x_0 \in X_0$ and the resulting controlled trajectory $x_0, x_1, x_2, \ldots, $ of \req{dynamics2} induced by $C$, the following statement holds: there exist $N, \ k_N \in \mathbb{N}_{\geq 0}$, such that 
\begin{enumerate}
\renewcommand{\labelenumi}{(C\arabic{enumi})}
\setlength{\leftskip}{0.1cm}
\item $x_{k_N} \in X_F$; 
\item $x_k \in X_S,\ \forall k\in\mathbb{N}_{0: k_N}$.  \qedwhite
\end{enumerate}
\end{mydef}

\begin{figure}[t]
  \begin{center}
   \includegraphics[width=6.0cm]{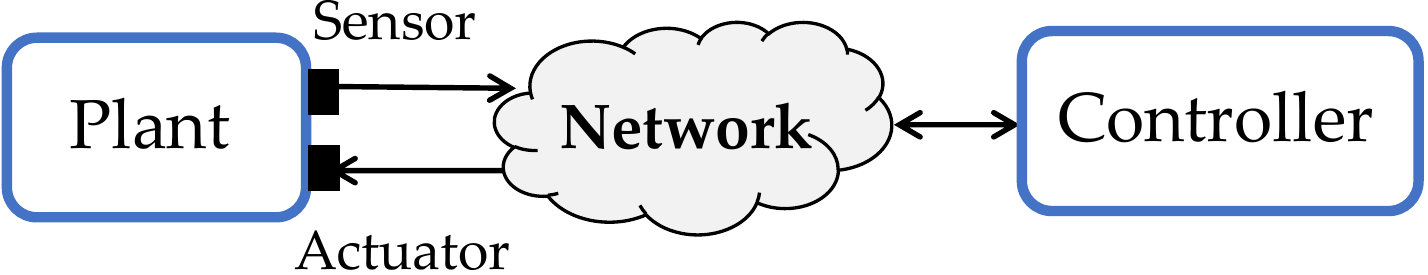}
   \caption{Networked Control System. } 
   \label{NCS}
   \vspace{-0.5cm}
  \end{center}
\end{figure}
 
That is, the controller $C$ is valid if every initial state can be steered to the target set in finite time (C1), while at the same time remaining in the safety set for all times (C2). 
We now state the main problem to be solved in this paper: { 
\begin{mypro}\label{problem}
Given $(X_S, X_F)$, synthesize a valid controller for the system \req{dynamics2} with the specification $(X_S, X_F)$. \qedwhite
\end{mypro}}

\section{Constructing symbolic models} 
As the first step to solve \rpro{problem}, this section presents a framework to obtain a symbolic model, which represents an abstracted expression of the control system. 

\subsection{Induced transition systems}\label{induced_sec}
Let us first define the notion of \textit{transition systems}, which captures the control system described in \req{dynamics2}: 
\begin{mydef}\label{plantsystem}
A \textit{transition system} induced by the system \req{dynamics2} is a tuple $\Sigma = (X, X_0, U, G, O)$, where: 
\begin{itemize}
\item $X = \mathbb{R}^{n_x}
$ is a set of states; 
\item $X_0 \subset X$ is a set of initial states; 
\item $U \subset \mathbb{R}^{n_u}$ is a set of inputs;
\item $G : X \times U \rightarrow 2^X$ is a transition map, where $x^+ \in G (x, u)$ iff $x^+ = f (x, u)$; 
\item $O : X \times U \rightarrow 2^{X^*}$ is an output map, where $\{ x^+ \} \in O (x, u)$ iff $ x^+ \in G (x, u)$.  \qedwhite
\end{itemize} 
\end{mydef}

In \rdef{plantsystem}, the transition $x^+ \in G (x, u)$ means that the system evolves from $x$ to $x^+$ by applying the control input $u$ according to \req{dynamics2}.  
The output map $O$ is defined to produce, for each transition $ x^+ \in G (x, u)$, the corresponding output, which is simply here given by $\{ x^+ \}$.  {Here, even though the system \req{dynamics2} is deterministic, the transition and the output maps are given by \textit{set-valued} maps, in order to show that \req{dynamics2} can be captured within the general definition of transition systems.}

Next, we extend \rdef{plantsystem} by incorporating the self-triggered strategy. 
As stated in \rsec{self_triggered_sec}, in the self-triggered strategy both the control input and the inter-communication time step are the decision variables that are determined by the controller, and the control input is applied constantly until the next communication time. In this paper, we express this fact by introducing the \textit{augmented transition system} of $\Sigma$, which is formally defined below: 

\begin{mydef}\label{switchedsystems}
\normalfont
Let $\Sigma = (X, X_0, U, G, O)$ be a transition system defined in \rdef{plantsystem}. 
The \textit{augmented transition system} of $\Sigma$ is a tuple $\Sigma_{A} = (X, X_0, U, M, G_{A}, O_A)$, where: 
\begin{itemize}
\item $X$ is a set of states; 
\item $X_0$ is a set of initial states; 
\item $U$ is a set of control inputs; 
\item ${M} = \mathbb{N}_{>0}$ is a set of inter-communication time steps; 
\item ${G}_{A} :  X \times U \times M \rightarrow 2^X$ is a transition map, where $x^+ \in {G}_{A} (x, u, m)$ iff $x^+ = \phi (x, u, m)$; 
\item $O_A : X \times U \times M \rightarrow 2^{X^*}$ is an output map, where $\{x_1, \ldots, x_{m} \} \in O_A (x, u, m)$, iff $x_p \in {G}_{A} (x, u, p)$, $\forall p \in \mathbb{N}_{1:m}$. 
\qedwhite 
\end{itemize} 
\end{mydef}

In essence, the augmented transition system $\Sigma_{A}$ introduces the new transition map ${G}_{A}$ and the output map ${O}_{A}$, which incorporate both the control input and the inter-communication time step as the decision variables. 
Specifically, $x^+ \in {G}_{A} (x, u, m)$ means that the system evolves from $x$ to $x^+$ by applying $u\in U$ constantly for $m$ time steps according to \req{dynamics2}. In addition, $\{x_1, \ldots, x_{m} \} \in O_A (x, u, m)$ means that, while applying $u$ constantly for $m$ time steps from $x$, it produces the output $\{x_1, \ldots, x_{m} \}$ as the corresponding sequence of states generated according to the transition map ${G}_{A}$. 

We say that ${x}_{0}, {x}_{1}, \ldots \in {X}$ is a \textit{trajectory} of ${\Sigma}_{A}$, if $x_0 \in X_0$ and there exist ${m}_0, {m}_1, \ldots \in {M}$ and ${u}_{k_0}, {u}_{k_1}, \ldots \in {U}$, such that $\{x_{k_\ell+1}, \ldots, x_{k_{\ell+1}}\} \in O_A (x_{k_\ell}, u_{k_\ell}, m_\ell)$, 
$\forall \ell \in \mathbb{N}_{\geq 0}$, where $k_0 = 0$, $k_{\ell+1} = k_\ell + {m}_\ell$, $\forall \ell \in \mathbb{N}_{\geq 0}$.  {Given $C$, we say that a sequence ${x}_{0}, {x}_{1}, \ldots \in {X}$ is a \textit{controlled} trajectory of ${\Sigma}_{A}$ induced by $C$}, if it is a trajectory of ${\Sigma}_{A}$ with $\{{u}_{k_\ell}, {m}_\ell \} \in C ( {x}_{k_\ell})$, $\forall \ell \in \mathbb{N}_{\geq 0}$. Note that the system \req{dynamics2} and ${\Sigma}_{A}$ produce the same controlled trajectories, i.e., if $x_0, x_1, \ldots \in X$ is a controlled trajectory of the system \req{dynamics2} induced by $C$, so is of ${\Sigma}_A$ induced by $C$, and vice versa. 

\subsection{Constructing symbolic models}
Having defined the augmented transition system, 
we now construct a {symbolic model} of $\Sigma _A$.  The symbolic model will be denoted as $\widetilde{\Sigma}^{\varepsilon} _{A} = (\widetilde{X}, \widetilde{X}_0, \widetilde{U}, \widetilde{M}, \widetilde{G}_{A}, \widetilde{O}_{A})$, where $\widetilde{X}$ is a set of states, $\widetilde{X}_0$ is a set of initial states, $\widetilde{U}$ is a set of inputs, $\widetilde{M}$ is a set of inter-communication time steps, $\widetilde{G}_{A}$ is a transition map, and $\widetilde{O}_{A}$ is an output map. 
A more formal definition of $\widetilde{\Sigma}^\varepsilon _{A}$ will be given later in this section. To derive the symbolic model, we make use of the notion of 
 {\textit{strong approximate alternating simulation relation}}\cite{borri18a}: 

\begin{mydef}[{Strong $\varepsilon$-ASR}]\label{asr2}
\normalfont
Let $\Sigma_{A} = (X, X_{0}, {U}, $ $M, {G}_{A}, O_A)$ and $\widetilde{\Sigma}^\varepsilon _{A} = (\widetilde{X}, \widetilde{X}_0, \widetilde{U}, \widetilde{M}, \widetilde{G}_{A}, \widetilde{O}_{A})$ be two transition systems. A relation $R (\varepsilon) \subseteq \widetilde{X} \times {X}$ with $\varepsilon \in \mathbb{R}_{\geq 0}$ is called a  {\textit{strong $\varepsilon$-approximate Alternating Simulation Relation} (strong $\varepsilon$-ASR for short)}
from $\widetilde{\Sigma}^\varepsilon _{A}$ to ${\Sigma}_{A}$, if the following conditions hold: 

\begin{enumerate}
\renewcommand{\labelenumi}{(D\arabic{enumi})}
\setlength{\leftskip}{0.1cm}
\item For every $\tilde{x}_0 \in \widetilde{X}_0$, there exists ${x}_0 \in {X}_0$ such that $(\tilde{x}_0, {x}_0) \in R (\varepsilon)$; 
\item For every $(\tilde{x}, x) \in R (\varepsilon)$, we have $\| \tilde{x} - x \| \leq \varepsilon$; 
\item For every $(\tilde{x}, x) \in R (\varepsilon)$ and for every $\tilde{u} \in \widetilde{U}$, $\widetilde{m} \in \widetilde{M}$,  {there exist $u = \widetilde{u} \in U$, $m = \widetilde{m} \in M$}, such that the following holds:
$\{x_1, \ldots, x_{m}\} \in O_A(x, u, m)$ implies the existence of $\{\tilde{x}_1, \ldots, \tilde{x}_{\widetilde{m}}\} \in \widetilde{O}_A(\tilde{x}, \tilde{u}, \widetilde{m})$, such that $(\tilde{x}_p, x_p) \in R (\varepsilon)$, $\forall p \in \mathbb{N}_{1:m} (= \mathbb{N}_{1:\widetilde{m}})$. \qedwhite 
\end{enumerate}
\end{mydef}

The notion of \textit{strong} $\varepsilon$-ASR is {strong} with respect to the standard notion of $\varepsilon$-ASR \cite{zamani2012a}, in the sense that it must ensure the existence of the same control input $u = \widetilde{u}$ (and $m = \widetilde{m}$) for the similarity condition of (D3). The concept of (strong) $\varepsilon$-ASR is useful for the controller synthesis in the following sense. Suppose that $\widetilde{\Sigma}^\varepsilon _{A}$ is constructed to guarantee the existence of a (strong) $\varepsilon$-ASR from $\widetilde{\Sigma}^\varepsilon _{A}$ to ${\Sigma}_{A}$. 
Then, it is shown that, the existence of a valid controller for $\widetilde{\Sigma}^\varepsilon _{A}$ implies the existence of a valid controller for ${\Sigma} _{A}$. In other words, a controller for the symbolic model that satisfies reachability and safety specifications, can be \textit{refined} to a controller for $\Sigma_A$ and the system \req{dynamics2} that guarantees the same specifications (for a detailed discussion, see \rsec{selftrig_sec}). 

In the following, we provide an approach to construct $\widetilde{\Sigma}^\varepsilon _{A}$ as the abstraction of $\Sigma_{A}$. Given $\varepsilon \in \mathbb{R}_{\geq 0}$, the symbolic model is constructed to guarantee the existence of a strong $\varepsilon$-ASR 
from $\widetilde{\Sigma}^\varepsilon _{A}$ to $\Sigma_{A}$. 
First, let $\widetilde{X}_0\subseteq X_0$, $\widetilde{X} \subseteq X$, $\widetilde{U} \subseteq U$ be given by $\widetilde{X}_0 = [X_0]_{\eta_x}, \widetilde{X} = [X]_{\eta_x}, \widetilde{U} = [U]_{\eta_u}$ for given $\eta_x, \eta_u \in \mathbb{R}_{>0}$. That is, we quantize the sets $X_0$, $X$, $U$ with the quantization parameters $\eta_x, \eta_u$. 
Moreover, let $\widetilde{M} \subset M$ be given by $\widetilde{M} = \mathbb{N}_{1:M_{\max}}$ for a given $M_{\max} \in \mathbb{N}_{>0}$. That is, we restrict that the inter-communication time step cannot exceed $M_{\max}$. The above quantization as well as the restriction on the inter-communication time steps will ensure that the controller synthesis algorithm can be terminated with a finite number of iterations. 
Based on the above definitions, the symbolic model $\widetilde{\Sigma}^\varepsilon _{A}$ is formally defined as follows: 

\begin{mydef}\label{symbolic_model}
\normalfont
Let $\Sigma_{A} = (X, X_0, U, M, G_{A}, O_A)$ be the augmented transition system of $\Sigma$ and let $\eta_x, \eta_u \in \mathbb{R}_{>0}$ and $M_{\max} \in \mathbb{N}_{>0}$ be the quantization parameters and the maximum inter-communication time step, respectively. For a given precision $\varepsilon \geq \eta_x$, a \textit{symbolic model} of $\Sigma_{A}$ is a tuple $\widetilde{\Sigma}^\varepsilon _{A} = (\widetilde{X}, \widetilde{X}_0, \widetilde{U}, \widetilde{M}, \widetilde{G}_{A}, \widetilde{O}_A)$, where 
\begin{itemize}
\item $\widetilde{X}$ is a set of states; 
\item $\widetilde{X}_0$ is a set of initial states;
\item $\widetilde{U}$ is a set of inputs; 
\item $\widetilde{M}$ is a set of inter-communication time steps; 
\item $\widetilde{G}_{A} : \widetilde{X} \times \widetilde{U} \times \widetilde{M} \rightarrow  2^{\widetilde{X}}$ is a transition map, where $\tilde{x}^+ \in \widetilde{G}_{A} (\tilde{x}, \tilde{u}, \widetilde{m})$ iff $\tilde{x}^+  \in 
{\cal B}_{\varepsilon_{\widetilde{m}}} ({\phi} (\tilde{x}, \tilde{u}, \widetilde{m}))$ with $\varepsilon_{\widetilde{m}} = L_x^{\widetilde{m}}  \varepsilon + \eta_x$; 
\item $\widetilde{O}_A : \widetilde{X} \times \widetilde{U} \times \widetilde{M} \rightarrow 2^{\widetilde{X}^*}$ is an output map, where $\{\tilde{x}_1, \ldots, \tilde{x}_{\widetilde{m}} \} \in \widetilde{O}_A (\tilde{x}, \tilde{u}, \widetilde{m})$ iff $\tilde{x}_{p} \in \widetilde{G}_{A} (\tilde{x}, \tilde{u}, p)$, $\forall p\in\mathbb{N}_{1:\widetilde{m}}$. \qedwhite 
\end{itemize} 
\end{mydef}

In contrast to \rdef{switchedsystems}, the symbolic model deals with the states in $\widetilde{X}$ and control inputs in $\widetilde{U}$, and introduces the new transition and the output maps $\widetilde{G}_{A}$, $\widetilde{O}_{A}$. 
As shown in \rdef{symbolic_model}, we have $\tilde{x}^+ \in \widetilde{G}_{A} (\tilde{x}, \tilde{u}, \widetilde{m})$ iff $\tilde{x}^+  \in 
{\cal B}_{\varepsilon_{\widetilde{m}}} ({\phi} (\tilde{x}, \tilde{u}, \widetilde{m}))$. 
Intuitively, the set ${\cal B}_{\varepsilon_{\widetilde{m}}}({\phi} (\tilde{x}, \tilde{u}, \widetilde{m}))$ represents the over-approximation of the reachable states from ${\cal B}_{\varepsilon}(\tilde{x})$ plus the quantization error $\eta_x$, by applying $\tilde{u}$ constantly for $\widetilde{m}$ time steps. The output map $\widetilde{O}_{A}$ is defined to produce the set of sequence of states according to the transition map $\widetilde{G}_{A}$. 
The existence of a strong $\varepsilon$-ASR from $\widetilde{\Sigma}^\varepsilon _{A}$ to $\Sigma_{A}$, can be justified by the following result: 
\begin{mylem}\label{asrresult}
Let $\Sigma_{A} = (X, X_0, U, M, G_{A}, O_A)$ be the augumented transition system of $\Sigma$ and let $\eta_x, \eta_u \in \mathbb{R}_{>0}$ and $M_{\max} \in \mathbb{N}_{>0}$ be the quantization parameters and the maximum inter-communication time step, respectively. Also, let $\widetilde{\Sigma}^\varepsilon _{A} = (\widetilde{X}, \widetilde{X}_0, \widetilde{U}, \widetilde{M}, \widetilde{G}_{A}, \widetilde{O}_A)$ be the symbolic model of $\Sigma_{A}$ defined in \rdef{symbolic_model}, where $\varepsilon$ is the precision parameter with $\varepsilon \geq \eta_x$. Then, the relation 
\begin{align}\label{relation}
R ( \varepsilon ) = \{ (\tilde{x}, x) \in \widetilde{X} \times X \ |\ \| \tilde{x} - x\| \leq \varepsilon \} 
\end{align}
is a strong $\varepsilon$-ASR from $\widetilde{\Sigma}^\varepsilon _{A}$ to ${\Sigma} _{A}$. \qedwhite 
\end{mylem}
\begin{proof}
Let $R(\varepsilon)$ be given by \req{relation} with $\varepsilon \geq \eta_x$. Since $\widetilde{X}_0 \subseteq X_0$, for every $\tilde{x}_0 \in \widetilde{X}_0$ there exists ${x}_0 = \tilde{x}_0 \in X_0$ such that $\|\tilde{x}_0 - x_0 \| = 0 \leq \varepsilon$. Hence, the condition (D1) in \rdef{asr2} holds. 
The condition (D2) is satisfied from \req{relation}. 
To check (D3), consider any $(\tilde{x}, x) \in R(\varepsilon)$ and $(\tilde{u}, \widetilde{m}) \in \widetilde{U} \times \widetilde{M}$. Let $u = \tilde{u} \in U$, $m= \widetilde{m} \in M$ and consider $\{x_1, \ldots, x_{m} \} \in O_A (x, u, m)$. This implies from \rdef{switchedsystems} that $x_p = \phi (x, u, p)$, $\forall p \in \mathbb{N}_{1:m}$. Now, pick $\tilde{x}_p \in \mathsf{Nearest}_{\widetilde{X}} (x_p)$, $\forall p\in \mathbb{N}_{1:m}$. Since $\widetilde{X} = [X]_{\eta_x}$, it follows that $\| \tilde{x}_p - x_p\| \leq \eta_x$, $\forall p\in \mathbb{N}_{1:m}$, i.e., $(\tilde{x}_p , x_p) \in R (\varepsilon)$, $\forall p \in \mathbb{N}_{1:m} (= \mathbb{N}_{1:\widetilde{m}})$. 
Then, we have 
\begin{align}
\| \tilde{x}_p - \phi(\tilde{x}, \tilde{u}, p) \| &\leq  \|x_p - \phi(\tilde{x}, \tilde{u}, p)\| + \eta_x \\
&\leq L_x^p \| x- \tilde{x} \| + \eta_x \leq  L_x^p \varepsilon + \eta_x, \notag 
\end{align}
for all $p\in \mathbb{N}_{1:\widetilde{m}}$, where we used $u = \tilde{u}$ and \rlem{liplem}. Hence, $\tilde{x}_p \in {\cal B}_{\varepsilon_p} (\phi(\tilde{x}, \tilde{u}, p))$ with $\varepsilon_p = L_x^p \varepsilon + \eta_x$, $\forall p \in \mathbb{N}_{1:\widetilde{m}}$ and thus $\tilde{x}_p \in \widetilde{G}_A (\tilde{x}, \tilde{u}, p)$, $\forall p\in \mathbb{N}_{1:\widetilde{m}}$. Thus, $\{\tilde{x}_1, \ldots, \tilde{x}_{\widetilde{m}}\} \in \widetilde{O}_A(\tilde{x}, \tilde{u}, \tilde{m})$ and $(\tilde{x}_p , x_p) \in R (\varepsilon)$, $\forall p \in \mathbb{N}_{1:m} (= \mathbb{N}_{1:\widetilde{m}})$. 
Therefore, $R(\varepsilon)$ is a strong $\varepsilon$-ASR from $\widetilde{\Sigma}^\varepsilon _{A}$ to ${\Sigma} _{A}$. 
\end{proof}

We say that $\tilde{x}_{0}, \tilde{x}_{1}, \ldots \in \widetilde{X}$ is a \textit{trajectory} of $\widetilde{\Sigma}^\varepsilon _{A}$, if $\tilde{x}_{0} \in \widetilde{X}_0$ and there exist $\widetilde{m}_0, \widetilde{m}_1, \ldots \in \widetilde{M}$ and $\tilde{u}_{k_0}, \tilde{u}_{k_1}, \ldots \in \widetilde{U}$, such that $\{\tilde{x}_{k_\ell+1}, \ldots, \tilde{x}_{k_{\ell+1}} \} \in \widetilde{O}_A (\tilde{x}_{k_\ell}, \tilde{u}_{k_\ell}, \widetilde{m}_\ell)$, $\forall \ell \in \mathbb{N}_{\geq 0}$, where $k_0 = 0$, $k_{\ell+1} = k_\ell + \widetilde{m}_\ell$, $\forall \ell \in \mathbb{N}_{\geq 0}$. 

\section{Self-triggered controller synthesis}\label{selftrig_sec}
In this section, we present a concrete algorithm to synthesize a valid controller for the system \req{dynamics2} as a solution to \rpro{problem}. 
Given $\Sigma_{A}$, suppose that $\widetilde{\Sigma}^\varepsilon _{A}$ is constructed such that $R(\varepsilon)$ is a strong $\varepsilon$-ASR from $\widetilde{\Sigma}^\varepsilon _{A}$ to $\Sigma_{A}$, where $R(\varepsilon)$ is defined in \req{relation}. 
In what follows, we first synthesize a controller for $\widetilde{\Sigma}^\varepsilon _{A}$. Then, we synthesize a controller for $\Sigma_A$ as well as for the system \req{dynamics2}, based on the controller for $\widetilde{\Sigma}^\varepsilon _{A}$. 

As with \req{controller}, let a controller for $\widetilde{\Sigma}^\varepsilon _{A}$ be given by 
$\widetilde{C} : \widetilde{X} \rightarrow 2^{\widetilde{U} \times \widetilde{M}}$. 
 {Given ${\widetilde{C}}$, we say that the sequence $\tilde{x}_{0}, \tilde{x}_{1}, \ldots \in \widetilde{X}$ is a \textit{controlled} trajectory of $\widetilde{\Sigma}^\varepsilon _{A}$ induced by  $\widetilde{C}$}, if it is a trajectory of $\widetilde{\Sigma}^\varepsilon _{A}$ with $\{\tilde{u}_{k_\ell}, \widetilde{m}_\ell \} \in \widetilde{C} ( \tilde{x}_{k_\ell})$, $\forall \ell \in \mathbb{N}_{\geq 0}$. 
Moreover, let $\widetilde{X}_S = [\mathsf{Int}_{\varepsilon} (X_S)]_{\eta_x}$ and  $\widetilde{X}_F = [\mathsf{Int}_{\varepsilon} (X_F)]_{\eta_x}$. 
This implies that 
\begin{align}\label{property}
\!\! \tilde{x} \in \widetilde{X}_S ({\rm or}\ \widetilde{X}_F), (\tilde{x}, x) \in R(\varepsilon) {\small \implies} x \in X_S({\rm or}\ {X}_F). 
\end{align}
Note that the sets $\widetilde{X}_S$, $\widetilde{X}_F$ are both finite, since $X_S$ and $X_F$ are both compact. 
The validity of the controller ${\widetilde{C}}$ for $\widetilde{\Sigma}^\varepsilon _{A}$ with $(\widetilde{X}_S, \widetilde{X}_F)$ is defined in the same way as \rdef{validity}. 

\begin{algorithm}[t]\label{overall_alg}
{\small 
\SetKwInOut{Input}{Input}
\SetKwInOut{Output}{Output}
\Input{$\widetilde{\Sigma}^{\varepsilon}_A$, $\widetilde{X}_S$, $\widetilde{X}_F$} 
\Output{${\cal L} (\tilde{x})$, $\forall \tilde{x} \in \widetilde{X}_S$, $\widetilde{P}_S \subseteq \widetilde{X}_S$}
$n \leftarrow 0$, $P^{(n)} \leftarrow \widetilde{X}_F$; \\
${\cal L} (\tilde{x}) \leftarrow 0, \ \forall \tilde{x} \in \widetilde{X}_F$; \\
\While {$P^{(n+1)}  \neq P^{(n)}$} {
$P^{(n+1)} \leftarrow P^{(n)} \cup {\rm Pre} (P^{(n)})$;  \\ 
${\cal L} (\tilde{x}) \leftarrow n+1$, $\forall \tilde{x} \in P^{(n+1)}\backslash P^{(n)}$; \label{numcom} \\
$n \leftarrow n+1$; 
}
$\widetilde{P}_S \leftarrow P^{(n)}$,\ ${\cal L} (\tilde{x}) \leftarrow \infty$, $\forall \tilde{x} \in \widetilde{X}_S \backslash \widetilde{P}_S$; 
    \caption{{\small Controller synthesis for $\widetilde{\Sigma}^\varepsilon _{A}$.}} 
    }
\end{algorithm}

Now, consider the problem of synthesizing a valid controller for $\widetilde{\Sigma}^\varepsilon _{A}$ with the specification $(\widetilde{X}_S, \widetilde{X}_F)$. The controller for $\widetilde{\Sigma}^\varepsilon _{A}$ can be found by employing a reachability game \cite{tabuadabook2009}, which is summarized in \ralg{overall_alg}. In the algorithm, the map ${\rm Pre} : 2^{\widetilde{X}_S} \rightarrow 2^{\widetilde{X}_S}$ is given by 
\begin{align}
{\rm Pre}& (P) = \{ \tilde{x} \in \widetilde{X}_S\ |\ \exists (\tilde{u}, \widetilde{m}) \in \widetilde{U} \times \widetilde{M}: \notag \\ 
&\!\!\!\!\!\! \forall \tilde{x}^+ \in \widetilde{G}_A (\tilde{x}, \tilde{u}, \widetilde{m}),\ \tilde{x}^+ \in P,\ \widetilde{O}_A (\tilde{x}, \tilde{u}, \widetilde{m}) \subseteq {\widetilde{X}^*_S} \}. \label{pre} 
\end{align}
 {Roughly speaking, ${\rm Pre}(P)$ is the set of all states in $\widetilde{X}_S$, for which there exists a pair of control input and inter-communication time step, such that all the corresponding successors according to the transition map $\widetilde{G}_A$ are inside $P$ (i.e., $\tilde{x}^+ \in P$), while all the corresponding outputs are inside $\widetilde{X}^* _S$.}
As shown in the algorithm, we iteratively compute $P^{(n)}$, $n = 0, 1, \ldots$ until it converges to a fixed point set denoted as $\widetilde{P}_S$. Intuitively, $P^{(n)}$ is the set of all states in $\widetilde{X}_S$ that can reach $\widetilde{X}_F$ within $n$ transitions. 
The map ${\cal L} : \widetilde{X}_S \rightarrow \widetilde{M}$ is used to stack, for each state in $\widetilde{X}_S$, the number of transitions required to reach $\widetilde{X}_F$. 
\ralg{overall_alg} is guaranteed to terminate with a finite number of iterations, since $\widetilde{X}_S$, $\widetilde{X}_F$, $\widetilde{U}$ and $\widetilde{M}$ are all finite. Based on \ralg{overall_alg}, we construct the controller for $\widetilde{\Sigma}^\varepsilon _{A}$ as follows:  {
\begin{align}\label{controller_synthesis}
\widetilde{C} (\tilde{x}) = \{ (\tilde{u}, \widetilde{m}) \in \widetilde{U} \times \widetilde{M}\ |\ \forall \tilde{x}^+ \in \widetilde{G}_A (\tilde{x}, \tilde{u}, \widetilde{m}) :  \notag \\  {\cal L} (\tilde{x}^+) < {\cal L} (\tilde{x}) ,\ \widetilde{O}_A (\tilde{x}, \tilde{u}, \widetilde{m}) \subseteq {\widetilde{X}^* _S} \}, 
\end{align}}
for all $\tilde{x} \in \widetilde{P}_S$. 
The following result shows that the above controller is proven valid with a certain initial condition: 
\begin{mylem}\label{validity_symbolic}
Let $\widetilde{C}$ be the controller for $\widetilde{\Sigma}^\varepsilon _{A}$ as derived in \req{controller_synthesis}. Then, $\widetilde{C}$ is valid for $\widetilde{\Sigma}^\varepsilon _{A}$ with the specification 
$(\widetilde{X}_S, \widetilde{X}_F)$, if $\widetilde{X}_0 \subseteq \widetilde{P}_S$. 
\qedwhite
\end{mylem}
\begin{proof}
Suppose that $\widetilde{X}_0 \subseteq \widetilde{P}_S$ and let $\tilde{x}_{k_0} \in \widetilde{X}_0$ with $k_0 = 0$ be the initial state of $\widetilde{\Sigma}^{\varepsilon} _A$. 
Since $\tilde{x}_{k_0} \in \widetilde{P}_S$, there exists $N\in\mathbb{N}_{\geq 0}$ such that ${\cal L} (\tilde{x}_{k_0}) = N$. This means that $\tilde{x}_{k_0} \in P^{(N)} \backslash P^{(N-1)} \subseteq {\rm Pre} (P^{(N-1)})$, and, from \req{pre}, there exists $(\tilde{u}', \widetilde{m}') \in \widetilde{U}\times \widetilde{M}$ such that $\forall \tilde{x}^+ \in \widetilde{G}_A (\tilde{x}_{k_0}, \tilde{u}', \widetilde{m}'),\ \tilde{x}^+ \in P^{(N-1)}$ (i.e., ${\cal L} (\tilde{x}^+) \leq  N-1$), $\widetilde{O}_A (\tilde{x}_{k_0}, \tilde{u}, \widetilde{m}) \subseteq {\widetilde{X}^*_S}$. Hence, the controller in \req{controller_synthesis} is feasible, i.e., $\widetilde{C} (\tilde{x}_{k_0}) \neq \emptyset$. 
Thus, by applying $\{\tilde{u}_{k_0}, \widetilde{m}_0 \} \in \widetilde{C} (\tilde{x}_{k_0})$, we obtain $\tilde{x}_{k_1} \in P^{(N-1)}$ and $\tilde{x}_{k} \in \widetilde{X}_S, \forall k\in\mathbb{N}_{k_0+1:k_1}$, where $k_1 = k_0 + \widetilde{m}_0$ and $\{ \tilde{x}_{k_0+1}, \ldots \tilde{x}_{k_1} \} \in \widetilde{O}_A (\tilde{x}_{k_0}, \tilde{u}_{k_0}, \widetilde{m}_0)$. 
In summary, $\tilde{x}_{k_0} \in P^{(N)} \implies \tilde{x}_{k_1} \in P^{(N-1)}$ and $\tilde{x}_{k} \in \widetilde{X}_S$, $\forall k\in\mathbb{N}_{k_0:k_1}$. 
By recursively applying the above for $k_0, k_1, \ldots$, we obtain 
$\tilde{x}_{k_0} \in P^{(N)} \implies \tilde{x}_{k_1} \in P^{(N-1)} \implies \cdots \implies \tilde{x}_{k_N} \in P^{(0)} = \widetilde{X}_F$ and $\tilde{x}_{k} \in \widetilde{X}_S$, $\forall k\in\mathbb{N}_{k_0:k_N}$. That is, the trajectory achieves reachability and safety. 
Thus, the controller in \req{controller_synthesis} is valid if $\widetilde{X}_0 \subseteq \widetilde{P}_S$. 
\end{proof}

Now, we refine the controller for $\Sigma_A$ as well as for the system \req{dynamics2}, based on the controller as derived in \req{controller_synthesis}: 
\begin{align}\label{controller_original}
C (x) = \left \{ \widetilde{C} (\tilde{x}) \ | \ \tilde{x} \in \mathsf{Nearest}_{\widetilde{P}_S} (x) \right \}, \ \forall x\in P_S, 
\end{align}
where $P_S = \{ x \in X_S\ |\ \exists \tilde{x} \in \widetilde{P}_S,\ (\tilde{x}, x) \in R(\varepsilon) \}$.
The controller in \req{controller_original} means that, for each $x \in P_S$, we pick the closest points in $\widetilde{P}_S$ to $x$, and associate the corresponding pairs of the control input and the inter-communication time step. Note that we have $P_S \subseteq X_S$ since $\widetilde{P}_S \subseteq \widetilde{X}_S$ and $\widetilde{X}_S = [\mathsf{Int}_{\varepsilon} (X_S)]_{\eta_x}$. 
Similarly to \rlem{validity_symbolic}, the following result shows that the controller for the system \req{dynamics2} is proven valid if a certain initial condition is satisfied. 

\begin{mythm}\label{mainresult}
Let $C$ be the controller as in \req{controller_original}. Then, $C$ is valid for $\Sigma_A$ as well as for the system \req{dynamics2} with the specification $({X}_S, {X}_F)$, if ${X}_0 \subseteq {P}_S$. 
\qedwhite 
\end{mythm}

\begin{proof}
Suppose that ${X}_0 \subseteq {P}_S$ and let $x_{k_0} \in X_0$ with $k_0 = 0$ be the initial state of the system \req{dynamics2}. 
Pick $\tilde{x}_{k_0} \in  \mathsf{Nearest}_{\widetilde{P}_S} (x_{k_0})$. 
Since $x_{k_0} \in P_S$, it follows that $(\tilde{x}_{k_0}, x_{k_0}) \in R(\varepsilon)$. 
Moreover, since $\tilde{x}_{k_0} \in \widetilde{P}_S$ there exists an $N\in\mathbb{N}_{\geq 0}$ such that ${\cal L} (\tilde{x}_{k_0}) = N$. 
Now, pick $\{\tilde{u}_{k_0}, \widetilde{m}_0 \} \in \widetilde{C} (\tilde{x}_{k_0})$ and $\{{u}_{k_0}, {m}_0 \} \in {C} ({x}_{k_0})$ with ${u}_{k_0} = \tilde{u}_{k_0}$, ${m}_0 = \widetilde{m}_0$. 
Let ${x}_{k_{0}+1}, \ldots, {x}_{k_{1}}$ with $k_1 = k_0 + m_0$ be the controlled trajectory of $\Sigma_A$ by applying $u_{k_0}$ for $m_0$ steps, i.e., 
$\{x_{k_0+1}, \ldots, x_{k_{1}}\} \in O_A (x_{k_0}, u_{k_0}, m_0)$. 
 {Then, from the fact that $R(\varepsilon)$ is a strong $\varepsilon$-ASR from $\widetilde{\Sigma}^\varepsilon _A$ to $\Sigma_A$ and from \rdef{asr2}, there exists $\{\tilde{x}_{k_0+1}, \ldots, \tilde{x}_{k_{1}}\} \in \widetilde{O}_A(\tilde{x}_{k_0}, \tilde{u}_{k_0}, \widetilde{m}_0)$, such that $(\tilde{x}_{k}, x_k) \in R(\varepsilon)$, $\forall k\in\mathbb{N}_{k_0 +1 : k_1}$.} 
Moreover, from the proof of \rlem{validity_symbolic}, it follows that $\tilde{x}_{k_{1}} \in P^{(N-1)} \subseteq \widetilde{P}_S$ and $\tilde{x}_{k} \in \widetilde{X}_S$, $\forall k\in\mathbb{N}_{k_0:k_1}$. In summary, we obtain $(\tilde{x}_{k}, x_k) \in R(\varepsilon), \tilde{x}_{k} \in \widetilde{X}_S$, $\forall k\in\mathbb{N}_{k_0 : k_1}$, and $\tilde{x}_{k_{1}} \in P^{(N-1)}$. 
Thus, by recursively applying the above procedure for $k_0, k_1, \ldots$, it follows that $(\tilde{x}_{k}, x_k) \in R(\varepsilon)$, $\tilde{x}_k \in \widetilde{X}_S$, $\forall k\in\mathbb{N}_{k_0:k_N}$ and $\tilde{x}_{k_N} \in P^{(0)} = \widetilde{X}_F$. 
Since $(\tilde{x}_{k}, x_k) \in R(\varepsilon)$, $\forall k \in\mathbb{N}_{k_0:k_N}$ and by using \req{property}, it then follows that ${x}_k \in {X}_S$, $\forall k\in\mathbb{N}_{k_0:k_N}$ and 
$x_{k_N} \in {X}_F$. 
Hence, the controller in \req{controller_original} is valid for $\Sigma_A$ if ${X}_0 \subseteq {P}_S$. 
Moreover, since $\Sigma_A$ and the system \req{dynamics2} produce the same controlled trajectories (see \rsec{induced_sec}), the controller in \req{controller_original} is also valid for the system \req{dynamics2} if ${X}_0 \subseteq {P}_S$. 
\end{proof}
\begin{myrem}[{On the selection of $\eta_x, \eta_u$}] 
 {As $\eta_x, \eta_u$ are selected smaller, the symbolic $\widetilde{\Sigma}^\varepsilon _{A}$ will be more precise with respect to ${\Sigma} _{A}$. 
Intuitively, this implies that the condition $X_0 \subseteq P_S$ is more likely to be satisfied, and, therefore, we may increase the possibility to synthesize the controller. 
However, selecting smaller $\eta_x, \eta_u$ leads to a heavier computation load of synthesizing the controller, due to the evaluation of ${\rm Pre} (\cdot)$ in \req{pre}. Thus, users may carefully select $\eta_x, \eta_u$ by considering the trade-off between the possibility to synthesize the controller and the computation load for the controller synthesis.} \qedwhite 
\end{myrem}

\begin{myrem}
 {Even though $X_0 \nsubseteq P_S$ does not hold, one might be still interested in finding a \textit{subset} of initial states $X^* _0 \subseteq X_0$, such that every controlled trajectory from $x_0 \in X^* _0$ achieves reachability and safety. 
This subset can be derived by taking the intersection between $X_0$ and $P_S$, i.e., $X^* _0 = X_0 \cap P_S$.}  \qedwhite 
\end{myrem}

\begin{figure}[tbp]
  \begin{center}
   \includegraphics[width=7.0cm]{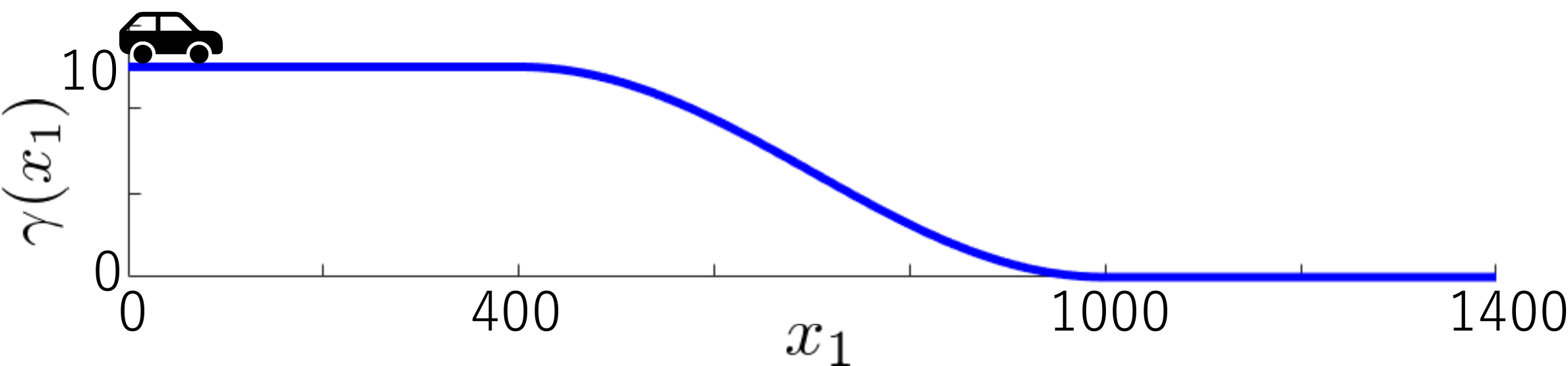}
   \caption{Geometry of the road based on the road grade profile $\alpha$.}
   \label{roadmap}
  \end{center}
  \vspace{-0.5cm}
 \end{figure}

\begin{table}[tbp]
\begin{center}
\caption{\small Parameter settings} \label{params}
{\small 
\begin{tabular}{l||l||l}
$\Delta: 1.0$ & $m_v$: 1000 & $O :  [400, 1000] \times [12, 18]$ \\
$g$: 9.8 & $U : [-500, 500]$ & $X_0 : [0, 50] \times[0, 18]$\\ 
$d_1$: 0.01 & $X_S : X_P \backslash O$ & $X_F : [1300, 1400]\times[0, 5]$ \\
$d_2$: 0.1 &  $X_P : [0, 1400] \times [0, 18]$
\end{tabular}}
\vspace{-0.5cm}
\end{center}
\end{table}

\section{Illustrative Example}
In this section we illustrate the effectiveness of the proposed approach through a simulation example. 
We consider the problem of controlling a vehicle on a road with a given road grade, which is expressed by the mapping $\alpha : [0, 1400] \rightarrow [-\pi/2, \pi/2]$, where $\alpha (x_1) = 0$, $\forall x_1 \in [0, 400)$, $\alpha (x_1) = -  \frac{\pi}{120}\sin \frac{x_1-400}{600} \pi$, $\forall x_1 \in [400, 1000)$, $\alpha (x_1) = 0$, $\forall x_1 \in [1000, 1400]$, where $x_1$ denotes the position of the vehicle. 
The geometry of the road based on the given $\alpha$ is illustrated in \rfig{roadmap}. In the figure, $\gamma (\cdot)$ denotes the elevation of the road obtained by $\gamma(x_1) = \gamma(0)+ \int^{x_1} _{0} \alpha (x) {\rm d}x$ with $\gamma(0) = 10$. Let $x_{k} = [x_{1,k}; x_{2,k}] \in \mathbb{R}^2$ be the state, where $x_{1,k}$ is the position and $x_{2,k}$ is the velocity of the vehicle at $k$. 
We assume that the dynamics of the vehicle is given by 
\begin{align}
{x}_{1,k+1} &= x_{1,k} + x_{2, k} \Delta, \\ 
{x}_{2,k+1} &=  {x_{2, k}} + \frac{1}{m_v} {\left (u_k - h (x_{1,k}, x_{2,k}) \right)}\Delta, 
\end{align}
where $u_k \in U$ is the force applied to the vehicle as the control input, $m_v$ is the mass of the vehicle, and $\Delta$ is the sampling time period. $h(x_1, x_2)$ is the nonlinear term given by $h(x_1, x_2) = d_1 x_2 + d_2 m_v g \cos (\alpha (x_1)) + m_v g \sin (\alpha(x_1))$, where $d_1$ is the aerodynamic drag coefficient, $d_2$ is the rolling force coefficient, and $g$ is the gravitational constant. The parameter settings as well as the initial, safety, target sets are illustrated in \rtab{params}. 

Based on the above setting, we define the symbolic model $\widetilde{\Sigma}^\varepsilon _A$ with $\varepsilon = \eta_x = 1$, $\eta_u = 50$, $M_{\max} = 30$, and \ralg{overall_alg} has been implemented to synthesize the controller. The execution time for the algorithm to be terminated is $20760$s on Windows 10, Intel(R) Core(TM) 2.40GHz, 8GB RAM. 
It has been shown that ${X}_0 \subseteq {P}_S$ and thus from \rthm{mainresult} the controller in \req{controller_original} is valid for the system \req{dynamics2}. \rfig{results2} (upper) illustrates some trajectories by applying the resulting controller in \req{controller_original}, and \rfig{results2} (lower) illustrates the applied control inputs from $x_0 = [25; 0]$. The figure shows that all trajectories enter $X_F$ while remaining in $X_S$ for all times, and, moreover, control inputs are updated aperiodically (with $6$ number of communication times) according to the derived self-triggered strategy. For comparisons, we have also implemented \ralg{overall_alg} with $M_{\max} = 1$ (i.e., periodic communication). The resulting trajectory from $x_0 = [25; 0]$ requires $109$ number of communication time steps to achieve reachability. 
Hence, the self-triggered controller (\rfig{results2}) achieves a more communication reduction than the periodic scheme, which shows the effectiveness of the proposed approach. 

\begin{figure}[tbp]
  \begin{center}
   \includegraphics[width=9cm]{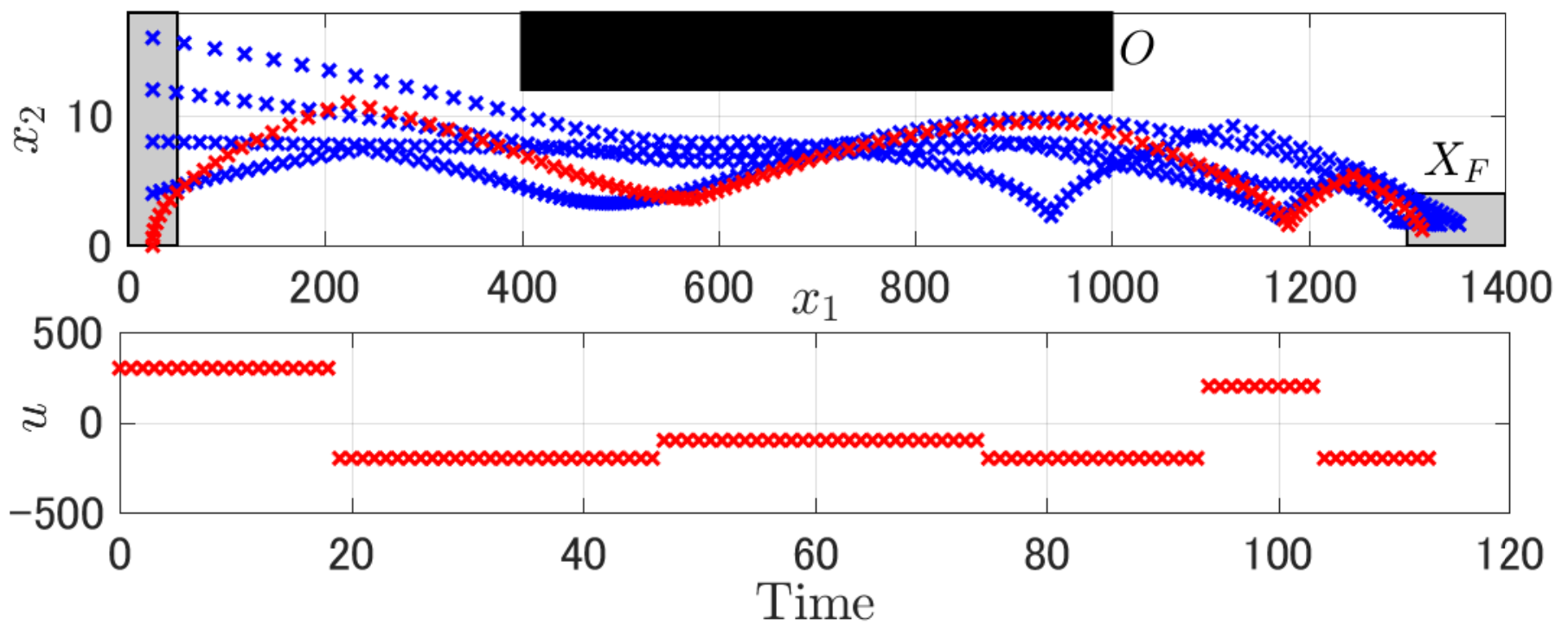}
   \caption{Simulation results. 
The upper figure illustrates some trajectories by applying \req{controller_original} (blue/red dotted lines). The lower figure illustrates the control inputs with $x_0 = [25; 0]$ (corresponding to the red trajectory in the upper figure).}
   \label{results2}
  \end{center}
  \vspace{-0.5cm}
 \end{figure}

\balance
\section{Conclusions and future work}
In this paper, we proposed a symbolic approach to synthesizing a self-triggered controller with reachability and safety specifications. 
The symbolic model was constructed based on the notion of approximate alternating simulation relations, and a controller was synthesized for the symbolic system via a reachability game, which was then refined to a controller for the original control system. 
 {Our future work involves providing an efficient algorithm to find suitable quantization parameters $\eta_x$, $\eta_u$, such that the condition ${X}_0 \subseteq {P}_S$ in \rthm{mainresult} is satisfied.}
Moreover, we would like to investigate improving the scalability of constructing symbolic models by employing compositional techniques, e.g., in \cite{adnane2018a,meyer2018b}. 



\end{document}